\documentclass[a4paper,11pt]{article}
\usepackage{amsmath,amsthm,amssymb}
\usepackage{enumerate}
\usepackage{url}

\usepackage[english]{babel}
\usepackage[utf8]{inputenc}
\usepackage[a4paper]{geometry}
\usepackage{latexsym}
\usepackage{amscd}
\usepackage{graphics}
\usepackage{color}
\usepackage{array}
\usepackage{mathrsfs}
\usepackage{graphicx}
\usepackage{stmaryrd}
\usepackage{mathabx}
\usepackage{dsfont}
\usepackage{comment}
\usepackage{calligra}
\usepackage{varioref}
\usepackage{prettyref}

\def\subjclass#1{{\renewcommand{\thefootnote}{}%
\footnote{\emph{Mathematics Subject Classification (2020):} #1}}}
\def\keywords#1{{\renewcommand{\thefootnote}{}%
\footnote{\emph{Keywords:} #1}}}
\def\ackn#1{{\renewcommand{\thefootnote}{}%
\footnote{#1}}}

\newtheorem{thm}{Theorem}[section]

\newtheorem{lem}[thm]{Lemma}

\newtheorem{prop}[thm]{Proposition}

\newtheorem{question}[thm]{Question}

\newrefformat{sec}{Section \ref{#1}}
\newrefformat{thm}{Theorem \ref{#1}}
\newrefformat{cor}{Corollary \ref{#1}}
\newrefformat{lem}{Lemma \ref{#1}}
\newrefformat{prob}{Problem \ref{#1}}
\newrefformat{prop}{Proposition \ref{#1}}
\newrefformat{conj}{Conjecture \ref{#1}}
\newrefformat{fact}{Fact \ref{#1}}
\newrefformat{question}{Question \ref{#1}}

\theoremstyle{definition}
\newtheorem{defin}[thm]{Definition}

\numberwithin{equation}{section}

\newcommand{\Z}{\mathbb{Z}}
\renewcommand{\H}{\mathbb{H}}
\newcommand{\C}{\mathbb{C}}

\newcommand{\R}{\mathbb{R}}

\newcommand{\A}{\mathcal{A}}
\newcommand{\cS}{\mathcal{S}}

\DeclareMathOperator{\Id}{Id}

\DeclareMathOperator{\GL}{GL}
\DeclareMathOperator{\NS}{NS}
\DeclareMathOperator{\Par}{Par}

\makeatletter
\newcommand{\alaligne}{~\vspace*{\topsep}\nobreak\@afterheading}
\makeatother

\newcommand{\cL}{\mathcal{L}}

\begin{document}

\title{Connected components of the general linear group of a real hereditarily indecomposable Banach space}

\author{N. de Rancourt}

\date{}

\maketitle

\subjclass{Primary: 47B01; Secondary: 46B03, 54D05.}

\ackn{The author was supported by FWF Grant P29999.}

\keywords{Hereditarily indecomposable Banach spaces, connectedness of general linear groups, complex structures on real Banach spaces.}


\begin{abstract}

We give a complete description of the structure of the connected components of the general linear group of a real hereditarily indecomposable Banach space, depending on the existence of complex structures on the space itself and on its hyperplanes. A side result is the fact that complex structures cannot exist simultaneously on such a space and on its hyperplanes.

\end{abstract}

\section{Introduction}

In this paper, unless otherwise specified, when speaking about a \textit{Banach space} (or simply a \textit{space}), we shall mean an infinite-dimensional Banach space, and by \textit{subspace} of a Banach space, we shall always mean infinite-dimensional, closed subspace.
By \textit{operator}, we shall always mean bounded linear operator. The algebra of operators on a normed space $X$ will be denoted by $\cL(X)$, and the general linear group of $X$, that is, the group of invertible elements of $\cL(X)$, will be denoted by $\GL(X)$. The algebra $\cL(X)$, and all of its subsets, will always be endowed with the topology induced by the operator norm.

\smallskip

Connected components of $\GL(X)$ are well-known when $X$ is a finite-dimensional normed space: this group is connected in the complex case, and has two connected components determined by the sign of the determinant in the real case. In the infinite-dimensional case, the situation is more complex. It follows from standard operator theory that $\GL(\ell_2)$ is connected in the complex case, and connectedness in the real case was first proved by Putnam and Wintner \cite{PutnamWintner}. Then, connectedness of the general linear group has been proved, both in the real and in the complex case, for the space $c_0$ independently by Arlt \cite{Arlt} and Neubauer \cite{Neubauer}, for the spaces $\ell_p$, $1\leqslant p < \infty$, by Neubauer \cite{Neubauer}, for $C([0, 1])$, $L_1([0, 1])$, $L_\infty([0, 1])$ and $\ell_\infty$ by Edelstein, Mityagin and Semenov \cite{EdelsteinMitjaginSemenov}, and for the spaces $L_p([0, 1])$, $1 < p < \infty$, by McCarthy and Mityagin (the proof first appeared in the survey \cite{MityaginSurvey} by Mityagin). On the other hand, it was proved by Douady \cite{Douady}, that $\GL(X \times Y)$, for $X$ and $Y$ being two distinct spaces among the $\ell_p$'s, $1 \leqslant p < \infty$, and $c_0$, is not connected, both in the real and in the complex case. For more details and results on this topic, and more generally on the homotopy structure of general linear groups of infinite-dimensional Banach spaces, we refer to Mityagin's survey \cite{MityaginSurvey}.

\smallskip

A common point between all of the latter examples is that their spaces of operators are very rich. In this paper, we study the case of some spaces where on the contrary, the space of operators is very poor: hereditarily indecomposable Banach spaces.

\begin{defin}
A Banach space is said to be \textit{hereditarily indecomposable (HI)} if it contains no topological direct sum of two subspaces.
\end{defin}

HI spaces were introduced by Gowers and Maurey in \cite{GowersMaureyHI}, where they built the first example of such a space (this space, existing in a real and in a complex version, will be called \textit{Gowers--Maurey's space} in this paper). It turns out that these spaces are very rigid. For example, in the complex case, the following result was proved by Gowers and Maurey in \cite{GowersMaureyHI}.

\begin{thm}[Gowers--Maurey]\label{thmGowersMaurey}
Every operator on a complex HI space $X$ has the form $\lambda \Id_X + S$, where $\lambda \in \C$ and $S$ is a strictly singular operator.
\end{thm}

Recall that an operator on $X$ is said to be \textit{strictly singular} if it induces no isomorphism between two subspaces of $X$. The set of strictly singular operators on $X$, that we will denote by $\cS(X)$ in this paper, is a two-sided, closed ideal of $\cL(X)$. An analogue of Theorem \ref{thmGowersMaurey} for real HI spaces, due to Ferenczi \cite{FerencziHI}, will be presented in Section \ref{secComplexStructures} (Theorem \ref{FerencziRCH}).

\smallskip

A consequence of the smallness of the space of operators on an HI space is that similar methods as those used in the finite-dimensional case apply quite well to the study of connected components of $\GL(X)$. This study was started by the author in \cite{RancourtHI}, were the complex case was settled.

\begin{thm}[de Rancourt]

If $X$ is a complex HI space, then $\GL(X)$ is connected.

\end{thm}

In the same paper, the author proved a partial result in the case of real HI spaces, that we present below. This result is stated in terms of \textit{homotopy} between operators. If $\mathcal{A} \subseteq \cL(X)$ and $S, T \in \A$, we say that $S$ and $T$ are \textit{homotopic} in $\mathcal{A}$ whenever there exists a continuous mapping $[0, 1] \to \mathcal{A}$, $t \mapsto T_t$, with $T_0 = S$ and $T_1 = T$. Since $\GL(X)$ is open in $\cL(X)$, it follows that two operators in $\GL(X)$ are homotopic in $\GL(X)$ if and only if they are in the same connected component of $\GL(X)$.

\smallskip

Also recall that an operator $R \in \cL(X)$ is a \textit{reflection} of $X$ if there exists $H \subseteq X$ a hyperplane and $x_0 \in X \setminus \{0\}$ such that $\forall x \in H \; R(x) = x$ and $R(x_0) = - x_0$. Call an \textit{antireflection} of $X$ an operator of the form $-R$, where $R$ is a reflection of $X$. It is easy to see that any two reflections of the same space $X$ are homotopic in $\GL(X)$, and that the same holds for antireflections. The partial result about real HI spaces proved in \cite{RancourtHI} is the following:

\begin{thm}[de Rancourt]\label{AtMostFour}
Let $X$ be a real HI space and let $T \in \GL(X)$. Then $T$ is homotopic, in $\GL(X)$, either to $\Id_X$, or to $-\Id_X$, or to all reflections, or to all antireflections.
\end{thm}

It immediately follows that $\GL(X)$ has at most four connected components. It is also remarked in \cite{RancourtHI} that this bound is attained in the case of the real Gowers--Maurey's space.

\smallskip

The goal of this paper is to give a complete description of the connected components of $\GL(X)$ when $X$ is a real HI space. We will prove, in particular, that $\GL(X)$ necessarily has either $2$ or $4$ connected components, and that both can happen (see Theorem \ref{main}). It turns out that the structure of connected components of $\GL(X)$ is closely related to the existence of complex structures on $X$ and on its hyperplanes, that has been widely studied by Ferenczi \cite{FerencziComplexStructures} and Ferenczi--Galego \cite{EvenSpaces}.

\smallskip

This paper is organized as follows. In Section \ref{secComplexStructures}, we review some properties of real HI spaces and of their complex structures, from papers \cite{FerencziHI, FerencziComplexStructures} by Ferenczi and \cite{EvenSpaces} by Ferenczi--Galego. At the end of the section, we state our main result, Theorem \ref{main}. The next two sections are devoted to its proof. In Section \ref{secReflections}, we prove that whenever $X$ is a real HI space, the identity of $X$ can never be homotopic to a reflection in $\GL(X)$ (Theorem \ref{thmReflexions}). In the course of the proof, some tools playing a similar role as the determinant in finite dimension will be introduced. In Section \ref{secDescription}, we will finish the proof of Theorem \ref{main} and, using the tools introduced in Section \ref{secReflections}, we will give a complete description of the connected components of $\GL(X)$ depending on the existence of complex structures on $X$ and on its hyperplanes (see Theorems \ref{DescriptionR}, \ref{DescriptionEven} and \ref{DescriptionOdd}). It will also be proved that there cannot simultaneously exist complex structures on a real HI space and on its hyperplanes (Theorem \ref{REvenOdd}).

\bigskip\bigskip

\section{Real HI spaces and complex structures}\label{secComplexStructures}

Theorem \ref{thmGowersMaurey} is not always true for real HI spaces. However, the folloing analogue was proved by Ferenczi in \cite{FerencziHI}:

\begin{thm}[Ferenczi]\label{FerencziRCH}

Let $X$ be a real HI space. Then the quotient algebra $\cL(X)/\cS(X)$ is either isomorphic, as a unitary Banach algebra, to $\R$, to $\C$, or to the quaternion algebra $\H$.

\end{thm}

We will say that the space $X$ has \textit{type $\R$}, \textit{type $\C$}, or \textit{type $\H$} depending on the case; the real Gowers--Maurey's space, for example, is of type $\R$. In the rest of this paper, we will denote by $\pi_X\colon \cL(X) \to \cL(X)/\cS(X)$ the projection, and the subalgebra of $\cL(X)/\cS(X)$ generated by the image of the identity will always be identified to $\R$. (Of course, for spaces of type $\C$ or $\H$, we cannot identify the whole algebra $\cL(X)/\cS(X)$ to $\C$ or $\H$ since this identification is in general not unique.) Observe that if $P \in \cL(X)$ is a projection onto a finite-codimensional subspace $Y$, then the mapping $\cL(Y) \to \cL(X)$, $T \mapsto T \circ P$ induces an isomorphism from $\cL(Y)/\cS(Y)$ onto $\cL(X)/\cS(X)$; in particular, a real HI space and its finite-codimensional subspaces have the same type.

\smallskip

If $(X, \|\cdot\|)$ is a real Banach space, we define a \textit{complex structure} on $X$ as a structure of complex vector space on $X$ extending its structure of real vector space, together with a complex norm $\|\cdot\|_\C$ which, when seen as a real norm, is equivalent to the original norm $\|\cdot\|$. Saying that a real Banach space $X$ admits a complex structure is hence equivalent to say that $X$ is isomorphic, as a real Banach space, to some complex Banach space; thus admitting a complex structure is an isomorphic property of real Banach spaces. We will identify two complex structures on $X$ when the associated complex norms are equivalent. Hence, a complex structure on $X$ is entierely determined by the action of the multiplication by $i$ on the space $X$; this multiplication should be an operator $I \in \cL(X)$ such $I^2 = -\Id_X$. Conversely, given such an operator $I$ on a real Banach space $X$, we can define a complex structure on $X$ by letting, for every $x \in X$, $ix := I(x)$ and $\|x\|_{\C} := \sup_{a \in \C, \,|a|=1} \|ax\|$. Thus, in the rest of this paper, we will identify the set of complex structures on $X$ with the set of $I \in \cL(X)$ such that $I^2 = -\Id_X$. 

\smallskip

If $X$ is a real Banach space, then $X$ admits a complex structure if and only if its $2$-codimensional subspaces do so. Indeed, if $X$ admits a complex structure $I$, then given any hyperplane $H \subseteq X$, it is easy to see that the subspace $H \cap I(H)$ of $X$ has codimension $2$ and is $I$-invariant; conversely, if a subspace $Y \subseteq X$ of codimension $2$ admits a complex structure, then $X$ is isomorphic to $Y \oplus \C$, so it admits a complex structure too. Thus, the problem of the existence of complex structures on finite-codimensional subspaces of $X$ is settled once we know whether $X$ and its hyperplanes admit complex structures. The following result of existence follows from results by Ferenczi and Galego \cite{EvenSpaces}.

\begin{thm}[Ferenczi--Galego]\label{ExistComplexStructures}
Let $X$ be a real Banach space and let $u \in \cL(X)/\cS(X)$ be such that $u^2 = -1$. Then exactly one of the following conditions holds:
\begin{enumerate}
    \item There exists $U \in \cL(X)$ with $\pi_X(U) = u$ such that $U^2 = -\Id_X$;
    \item There exists $U \in \cL(X)$ admitting an invariant hyperplane $H$, with $\pi_X(U) = u$, and such that $(U \restriction_H)^2 = -\Id_H$.
\end{enumerate}
\end{thm}

It follows that as soon as $\cL(X)/\cS(X)$ contains a square root of $-1$, then there exist a complex structure, either on $X$, or on its hyperplanes. (Of course, nothing prevents both to happen at the same time, since a complex structure on $X$ and a complex structure on its hyperplanes could come from distinct elements of $\cL(X)/\cS(X)$; this is the case, for instance, for $X = \ell_2$.) In particular if $X$ is a real HI space having type $\C$ or $\H$, then either $X$, or its hyperplanes, admits a complex structure. Conversely, it is immediate that if either $X$, or its hyperplanes, admit a complex structure, then some element of $\cL(X)/\cS(X)$ should have square $-1$. Thus, real HI spaces of type $\R$, and in particular the real Gowers--Maurey's space, admit no complex structures, neither their hyperplanes do.

\smallskip

The following notions were introduced by Ferenczi and Galego in the same paper \cite{EvenSpaces}.

\begin{defin}\label{DefEvenOdd}

\alaligne

\begin{enumerate}
    \item The real Banach space $X$ is said to be \textit{even} if it admits a complex structure, and its hyperplanes do not.
    \item The real Banach space $X$ is said to be \textit{odd} if it does not admit a complex structure, and its hyperplanes do.
\end{enumerate}

\end{defin}

This terminology comes from the finite-dimensional case: spaces of even dimension are even, and spaces of odd dimension are odd. It follows from the previous remarks that every real HI space $X$ satisfies exactly one of the four following properties:

\begin{enumerate}
    \item $X$ has type $\R$;
    \item $X$ is even;
    \item $X$ is odd;
    \item both $X$ and its hyperplanes admit a complex structure.
\end{enumerate}

We have examples in the three first cases. As already said, the real Gowers--Maurey's space has type $\R$. In \cite{FerencziComplexStructures}, Ferenczi built two examples of even real HI spaces, $X(\C)$ and $X(\H)$, having respectively type $\C$ and $\H$. Since hyperplanes of even spaces are obviously odd, and similarly, hyperplanes of odd spaces are even, it follows that hyperplanes of $X(\C)$ and of $X(\H)$ are odd HI spaces of respective types $\C$ and $\H$. On the other hand, we will show in Section \ref{secDescription}, as a byproduct of our work on general linear groups, that case 4. cannot possibly occur (see Theorem \ref{REvenOdd}). This solves by the negative the following question by Ferenczi and Galego (still open in the general case) in the special case of HI spaces:

\begin{question}[Ferenczi--Galego]
Does there exist a real Banach space $X$ which is not isomorphic to its hyperplanes, such that both $X$ and its hyperplanes admit a complex structure?
\end{question}

We are now ready to state our main theorem, which will be proved in the next two sections.

\begin{thm}\label{main}

Let $X$ be a real HI space. Then exactly one of the three following condition holds.
\begin{enumerate}
    \item $X$ has type $\R$, and $\GL(X)$ has exactly four connected components: one containing $\Id_X$, one containing $-\Id_X$, one containing all reflections, and one containing all antireflections.
    \item $X$ is even, and $\GL(X)$ has exactly two connected components: one containing $\Id_X$ and $-\Id_X$, and one containing all reflections and all antireflections.
    \item $X$ is odd, and $\GL(X)$ has exactly two connected components: one containing $\Id_X$ and all antireflections, and one containing $-\Id_X$ and all reflections.
\end{enumerate}

\end{thm}

In each of the latter cases, a characterization of operators belonging to each connected component will be given in Section \ref{secDescription}, see Theorems \ref{DescriptionR}, \ref{DescriptionEven} and \ref{DescriptionOdd}.

\bigskip\bigskip

\section{Reflections on real HI spaces}\label{secReflections}

In this section, we prove the following theorem.

\begin{thm}\label{thmReflexions}
Let $X$ be a real HI space. Then $\Id_X$ cannot be homotopic to a reflection in $\GL(X)$.
\end{thm}

In the proof of this theorem, we will need to use spectral theory, and hence to pass to the \textit{complexification} of the real space $X$. This complexification, denoted by $X^\C$, is defined as the set of formal sums $x+iy$ for $x, y \in X$, endowed with the complex vector space structure defined by $(a+ib)(x+iy) := (ax - by) + i(ay + bx)$ for $a, b \in \R$, and with the complex norm defined by $\|x + iy\| := \sup_{\theta \in \R} \|\cos \theta \cdot x + \sin \theta \cdot y \|$; this makes it a complex Banach space. Given $T \in \cL(X)$, we can define its \textit{complexification} $T^\C \in \cL(X^\C)$ by $T^\C(x+iy) := T(x)+iT(y)$. When talking about the spectrum of the operator $T$, we will actually abusively talk about the spectrum of the operator $T^\C$; the same convention will be adopted for all spectrum-related notions such as eigenvalues, multiplicities, etc. The spectrum of the operator $T$ will be denoted by $\sigma(T)$. It is easy to see that the spectrum of an operator $T$ on a real Banach space is always conjugation-invariant, and that if $\lambda$ is an eigenvalue of $T$, then so is $\overline{\lambda}$, with the same multiplicity. It is also straightforward that the complexification of a Fredholm operator is itself Fredholm, with the same index.

\begin{lem}\label{FiniteMultiplicity}

Let $X$ be a real HI space, and let $T \in \cL(X)$. Let $\lambda \in \sigma(T) \cap \R$ be such that $\pi_X(T) \neq \lambda$. Then $\lambda$ is an isolated eigenvalue of $T$ with finite multiplicity.

\end{lem}

\begin{proof}

Recall that the \textit{essential spectrum} of $T$ is the set of all $\mu \in \C$ such that $T^\C - \mu\Id_{X^\C}$ is not semi-Fredholm. It is well known that every operator on an HI space is either strictly singular, or Fredholm with index $0$ (see for example \cite{RancourtHI}, Theorem 3.4). It follows that $\lambda$ does not belong to the essential spectrum of $T$.

\smallskip

It was proved by Gowers and Maurey \cite{GowersMaureyHI} that the spectrum of an operator on a real HI space is countable. It follows from \cite{Kato}, Theorem 5.33, that every non-essential spectral value of $T$ is an isolated eigenvalue with finite multiplicity. This is, in particular, the case for $\lambda$.

\end{proof}

For $I$ an open interval of $\R$, denote by $\NS^I(X)$ the set of all $T \in \cL(X)$ satisfying the two following conditions:
\begin{enumerate}
    \item $\pi_X(T) \notin I$;
    \item $\sigma(T) \cap \partial I = \varnothing$.
\end{enumerate}
(Here, $\partial I$ denotes the boundary of $I$ in $\R$.) The notation $\NS$ stands for \textit{non-singular}. By Lemma \ref{FiniteMultiplicity}, for $T \in \NS^I(X)$, $\sigma(T)\cap I$ is made of finitely many isolated eigenvalues with finite multiplicities. We will denote by $p^I(T)$ the parity of the sum of the multiplicities of eigenvalues of $T$ that are contained in $I$. Here, by a \textit{parity}, we will mean an element of $\Z/2\Z$.

\smallskip

The next lemma, along with Lemma \ref{lemParity}, is reminiscent of the proof of Proposition 8 in \cite{EvenSpaces}.

\begin{lem}\label{ContinuityEigenvalues}

Let $X$ be a real HI space and $I$ be an open interval of $\R$. Then the mapping $p^I \colon \NS^I(X) \to \Z/2\Z$ is locally constant.

\end{lem}

To prove Lemma \ref{ContinuityEigenvalues}, we will need the following result of continuity of the spectrum. For a proof, see \cite{Kato}, Chapter Four, Subsection 3.5.

\begin{prop}\label{kato}
Suppose $T \colon Y \to Y$ is an operator on a complex space, and $\Gamma$ is a rectifiable, simple closed curve in $\C$ such that $\Gamma \cap \sigma(T) = \varnothing$. Denote by $V$ the bounded connected component of $\C \setminus \Gamma$, and suppose that $\sigma(T) \cap V$ consists in finitely many eigenvalues of $T$ with finite multiplicity. Then for every $S \in \cL(Y)$ close enough to $T$, $\sigma(S) \cap V$ consists in finitely many eigenvalues of $S$ and the sum of their multiplicities is equal to the sum of the multiplicities of the eigenvalues of $T$ in $V$.
\end{prop}

\begin{proof}[Proof of Lemma \ref{ContinuityEigenvalues}]

Fix $T \in \NS^I(X)$. Consider $\Gamma$ a conjugation-invariant, rectifiable, simple closed curve in $\C$, not intersecting $\sigma(T)$, such that, denoting by $V$ the bounded connected component of $\C \setminus \Gamma$, we have $V \cap \sigma(T) = I \cap \sigma(T)$, and $I \cap (-\|T^\C\| - 1, \|T^\C\| + 1) \subseteq V$. This last condition implies that for $S \in \NS^I(X)$ close enough to $T$, we have $\sigma(S) \cap I \subseteq V$.

\smallskip

Denote by $n$ the sum of the multiplicities of the eigenvalues of $T$ that are in $I$. By Proposition \ref{kato}, for $S \in \NS^I(X)$ close enough to $T$, the sum of the multiplicities of the eigenvalues of $S$ that are in $V$ is equal to $n$. By invariance of the spectrum of $S$ under conjugation, the sum of the multiplicities of the eigenvalues of $S$ that are in $V \setminus I$ is even. Thus, for such an $S$, the sum of the multiplicities of the eigenvalues of $S$ that are in $I$ has the same parity as $n$, as wanted.

\end{proof}

\begin{lem}\label{lemParity}
Let $X$ be a real HI space having type $\C$ or $\H$. Then $p^\R \colon \NS^\R(X) \to \Z/2\Z$ is constant.
\end{lem}

\begin{proof}
Let $S, T \in \NS^\R(X)$; we show that $p^\R(S) = p^\R(T)$. We first claim that one of the segments $[S, T]$ and $[S, -T]$ is contained in $\NS^\R(X)$. If not, then both segments $[\pi_X(S), \pi_X(T)]$ and $[\pi_X(S), -\pi_X(T)]$ intersect $\R$, so there exist $\lambda, \mu \in (0, 1)$ such that both $u := \lambda \pi_X(S) + (1 - \lambda)\pi_X(T)$ and $v := \mu \pi_X(S) - (1 - \mu)\pi_X(T)$ are elements of $\R$. Hence, $\frac{(1-\mu)u + (1-\lambda)v}{(1-\mu)\lambda + (1 - \lambda)\mu} = \pi_X(S) \in \R$, contradicting the assumption that $S \in \NS^\R(X)$.

\smallskip

Since $p^\R(T) = p^\R(-T)$, we can assume, without loss of generality, that the segment $[S, T]$, is contained in $\NS^\R(X)$. If follow that $S$ and $T$ are in the same connected component of $NS^\R(X)$. Since, by Lemma \ref{ContinuityEigenvalues}, $p^\R$ is locally constant on $\NS^\R(X)$, we get that $p^\R(S) = p^\R(T)$.
%
\end{proof}

As a consequence, if $X$ is a real HI space having type $\C$ or $\H$, we can define $\Par(X) \in \Z/2\Z$ as the unique value taken by $p^\R$ on $\NS^\R(X)$ (observe that in the case of a space of type $\R$, this definition would make no sense since in this case, $\NS^\R(X)$ is empty). $\Par(X)$ will be called the \textit{parity} of $X$. As it will turn out (see Proposition \ref{PropParity} below), this notion of parity coincindes with the one defined by Ferenczi and Galego in \cite{EvenSpaces} (Definition \ref{DefEvenOdd}).

\smallskip

Observe that, for $X$ a real HI space of type $\C$ or $\H$, $\NS^{(-\infty, 0)}(X)$ and $\NS^{(0, +\infty)}(X)$ form an open cover of $\GL(X)$. Their intersection is precisely $\NS^\R(X) \cap \GL(X)$, and for every $T \in \NS^\R(X)\cap \GL(X)$, we have $p^{(-\infty , 0)}(T) + p^{(0, + \infty)}(T) = p^\R(T) = \Par(X)$. So the following conditions correctly define a mapping $\Par \colon \GL(X) \to \Z/2\Z$:
\begin{itemize}
\item $\Par(T) := p^{(-\infty , 0)}(T)$ for $T \in \NS^{(-\infty, 0)}(X)$;
\item $\Par(T) := \Par(X) - p^{(0 , + \infty)}(T)$ for $T \in \NS^{(0 , + \infty)}(X)$.
\end{itemize}
$\Par(T)$ will be called the \textit{parity} of the operator $T$.

\begin{lem}\label{ContinuityParity}
Let $X$ be a real HI space having type $\C$ or $\H$. Then $\Par \colon \GL(X) \to \Z/2\Z$ is locally constant.
\end{lem}

\begin{proof}
It is enough to show that it is locally constant on the two elements of the open cover $\{\NS^{(-\infty, 0)}(X), \NS^{(0, +\infty)}(X)\}$ of $\GL(X)$. This directly follows from the definition of $\Par(T)$ on each of these open sets and from Lemma \ref{ContinuityEigenvalues}.
\end{proof}

\begin{proof}[Proof of Theorem \ref{thmReflexions}]

If the space $X$ has type $\C$ or $\H$, the result directly follows from Lemma \ref{ContinuityParity}, and for the fact that $\Par(\Id_X) = 0$ and $\Par(R) = 1$ for every reflection $R$ of $X$.

\smallskip

The case of spaces of type $\R$ has already been treated in \cite{RancourtHI}, but we reproduce the proof here for completeness of this paper. If $X$ has type $\R$, then $\pi_X \restriction_{\GL(X)}$ takes values in $\R \setminus \{0\}$. This implies that the image by $\pi_X$ of the connected component $C$ of the identity in $\GL(X)$ is entierely contained in $(0, +\infty)$; in other words, $C \subseteq \NS^{(-\infty, 0)}(X)$. Hence, by Lemma \ref{ContinuityEigenvalues}, $p^{(-\infty, 0)}$ is constant on $C$; and its value is necessarily $p^{(-\infty, 0)}(\Id_X) = 0$. Knowing that for a reflection $R$, we have $p^{(-\infty, 0)}(R) = 1$, this shows that no reflection belongs to $C$.

\end{proof}

\bigskip\bigskip

\section{A complete description of the connected components of $\GL(X)$}\label{secDescription}

In this section, using tools developped in Section \ref{secReflections}, we provide a complete description of the connected components of $\GL(X)$, for $X$ a real HI space. Our results will, in particular, imply Theorem \ref{main}. We start with some preliminary results.

\pagebreak

\begin{lem}\label{FromComplexStructuresToHomotopy}

Let $X$ be a real Banach space.
\begin{enumerate}
\item If $X$ admits a complex structure, then $\Id_X$ and $-\Id_X$ are homotopic in $\GL(X)$.
\item If hyperplanes of $X$ admits a complex structure, then $\Id_X$ is homotopic to antireflections in $\GL(X)$.
\end{enumerate}

\end{lem}

\begin{proof}

\begin{enumerate}

\item Fix a complex structure on $X$. Then $[0, 1] \to \GL(X)$,  $t \mapsto e^{i\pi t}\Id_X$ is a homotopy between $\Id_X$ and $-\Id_X$ in $\GL(X)$.

\item Fix $H$ a hyperplane of $X$ and $x_0 \notin H$. Fix a complex structure on $H$. For $t \in [0, 1]$, define $T_t \in \cL(X)$ by $T_t(x_0) := x_0$, and $T_t(h) := e^{i \pi t}h$ for $h \in H$. Then $[0, 1] \to \GL(X)$, $t \mapsto T_t$ is a homotopy between $T_0 = \Id_X$ and $T_1$ an antireflection.

\end{enumerate}

\end{proof}

\begin{thm}\label{REvenOdd}

Let $X$ be a real HI space. Then either $X$ is has type $\R$, or $X$ is even, or $X$ is odd, and these three cases are mutually exclusive.

\end{thm}

\begin{proof}
We have already seen in Section \ref{secComplexStructures} that the three cases are mutually exclusive, and that the only other possible case is that both $X$ and its hyperplanes admit complex structures. We show that this last case cannot happen. Suppose it does. Then by Lemma \ref{FromComplexStructuresToHomotopy}, $\Id_X$ is both homotopic, in $\GL(X)$, to $-\Id_X$ and to antireflections. This implies that $\Id_X$ is homotopic to reflections, contradicting Theorem \ref{thmReflexions}.
\end{proof}

\begin{prop}\label{PropParity}

Let $X$ be a real HI space having type $\C$ or $\H$. Then $X$ is even iff $\Par(X) = 0$, and is odd iff $\Par(X) = 1$.

\end{prop}

\begin{proof}

If $X$ is even, then $\Id_X$ and $-\Id_X$ are homotopic in $\GL(X)$ by Lemma \ref{FromComplexStructuresToHomotopy}. By Lemma \ref{ContinuityParity}, it follows that $\Par(\Id_X) = \Par(-\Id_X)$. We know that $\Par(\Id_X) = p^{(-\infty, 0)}(\Id_X) = 0$, and that $\Par(-\Id_X) = \Par(X) - p^{(0, +\infty)}(-\Id_X) = \Par(X)$, so it follows that $\Par(X) = 0$.

\smallskip

A similar argument shows that if $X$ is odd, then $\Par(X) = 1$. Since by Theorem \ref{REvenOdd}, every real HI space of type $\C$ or $\H$ is either even, or odd, we deduce that the two implications we just proved are actually equivalences.

\end{proof}

We now describe the structure of the connected components of $\GL(X)$, for $X$ a real HI space, in the three cases defined by Theorem \ref{REvenOdd}, in the three theorems below. Their combination, along with Theorem \ref{REvenOdd}, implies Theorem \ref{main}.

\begin{thm}\label{DescriptionR}

Let $X$ be a real HI space having type $\R$. Then $\GL(X)$ has exactly four connected components, listed below:
\begin{itemize}
\item $\GL_0^+(X) := \{T \in \GL(X) \mid \pi_X(T) > 0 \text{ and } p^{(- \infty, 0)}(T) = 0\}$, containing $\Id_X$;
\item $\GL_1^+(X) := \{T \in \GL(X) \mid \pi_X(T) > 0 \text{ and } p^{(- \infty, 0)}(T) = 1\}$, containing all reflections;
\item $\GL_0^-(X) := \{T \in \GL(X) \mid \pi_X(T) < 0 \text{ and } p^{(0, + \infty)}(T) = 0\}$ containing $-\Id_X$;
\item $\GL_1^-(X) := \{T \in \GL(X) \mid \pi_X(T) < 0 \text{ and } p^{(0, + \infty)}(T) = 1\}$, containing all antireflections.
\end{itemize}

\end{thm}

\begin{proof}
The fact that $\GL_0^+(X)$, $\GL_0^-(X)$, $\GL_1^+(X)$, and $\GL_1^-(X)$ contain respectively $\Id_X$, all reflections, $-\Id_X$, and all antireflections, comes directly from the definitions. In particular, these four sets are nonempty. Now observe that every connected component of $\GL(X)$ is entierely contained in one of these four sets. Indeed, if $C$ is such a connected component, then the continuity of $\pi_X$ shows that $C$ is either entirely contained in $\GL_0^+(X) \cup \GL_1^+(X)$, or in $\GL_0^-(X) \cup \GL_1^-(X)$; and if, for instance, we are in the first case, then the local constancy of $p^{(- \infty, 0)}$ that $C$ is either contained in $\GL_0^+(X)$, or in $\GL_1^+(X)$. Since, by Theorem \ref{AtMostFour}, $\GL(X)$ has at most four connected components, it follows that the sets $\GL_0^+(X)$, $\GL_0^-(X)$, $\GL_1^+(X)$, and $\GL_1^-(X)$ are exactly the connected components of $\GL(X)$.
\end{proof}

\begin{thm}\label{DescriptionEven}

Let $X$ be an even real HI space. Then $\GL(X)$ has exactly two connected components, listed below:
\begin{itemize}
    \item $\GL_0(X) := \{T \in \GL(X) \mid \Par(T) = 0\}$, containing $\Id_X$ and $-\Id_X$;
    \item $\GL_1(X) := \{T \in \GL(X) \mid \Par(T) = 1\}$, containing all reflections and all antireflections.
\end{itemize}

\end{thm}

\begin{proof}
The facts that $\GL_0(X)$ contains $\Id_X$ and $-\Id_X$, and that $\GL_1(X)$ contains all reflections and all antireflections, are direct consequences of the definition of these sets and of Proposition \ref{PropParity}. The local constancy of $\Par$ on $\GL(X)$ (see Lemma \ref{ContinuityParity}) shows that every connected component of $\GL(X)$ is entirely contained in one of these two sets. The combination of Lemma \ref{FromComplexStructuresToHomotopy} and Theorem \ref{AtMostFour} shows that $\GL(X)$ has at most two connected components. Those are necessarily $\GL_0(X)$ and $\GL_1(X)$.
\end{proof}

\begin{thm}\label{DescriptionOdd}

Let $X$ be an odd real HI space. Then $\GL(X)$ has exactly two connected components, listed below:
\begin{itemize}
\item $\GL_0(X) := \{T \in \GL(X) \mid \Par(T) = 0\}$, containing $\Id_X$ and all antireflections;
\item $\GL_1(X) := \{T \in \GL(X) \mid \Par(T) = 1\}$, containing $-\Id_X$ and all reflections.
\end{itemize}

\end{thm}

\begin{proof}
The proof is similar to this of Theorem \ref{DescriptionEven}.
\end{proof}

The two last results show that the parity of an operator plays a similar role as the sign of its determinant in finite-dimension. More precisely, if $X$ is a finite-dimensional normed space, then the mappings $p^I$, for $I$ an interval of $\R$, can be defined on the whole $\cL(X)$, and $\Par(X)$ and $\Par(T)$ for $T \in \GL(X)$ can thus be correctly defined. We have that for all $T \in \GL(X)$, $e^{i\pi \Par(T)}$ is exactly the sign of the determinant of $T$. Moreover, Proposition \ref{PropParity} and Theorems \ref{DescriptionEven} and \ref{DescriptionOdd} are still valid in this case, replacing \textit{even} and \textit{odd} by \textit{even-dimensional} and \textit{odd-dimensional}, respectively.

\smallskip

Some results of this paper have been independently proved by Maurey \cite{MaureyQuaternions}, in particular Theorems \ref{thmReflexions} and \ref{REvenOdd}. In this text, he also proves some results about the existence of quaternionic structures on a real HI space and on its finite-codimensional subspaces, using quite different methods.

\bigskip

\paragraph{Acknowledgments.} The author would like to thank Bernard Maurey for very interesting questions and discussions on this topic; research that led to this paper was initiated by one of his questions. The author would also like to thank Valentin Ferenczi for advice on this topic.

\bigskip\bigskip

\bibliographystyle{plain}
\bibliography{main}

  \par
  \bigskip
    \textsc{\footnotesize Universit\"at Wien, Fakult\"at f\"ur Mathematik, Kurt G\"odel Research Center, Oskar-Morgenstern-Platz 1, 1090 Wien, AUSTRIA}
    
    \smallskip
    
    {\small \textit{Email address:} \url{noe.de.rancourt@univie.ac.at}}

\end{document}